\documentclass[11pt,a4paper]{article}%
\usepackage{amssymb,amsmath,amsfonts,amsthm,array,bm,color}%
\usepackage{amsmath}%
\usepackage{amsfonts}%
\usepackage{amssymb}%
\usepackage{graphicx}
\providecommand{\U}[1]{\protect\rule{.1in}{.1in}}
\newtheorem{theorem}{Theorem}[section]
\newtheorem{lemma}[theorem]{Lemma}

\newtheorem{proposition}[theorem]{Proposition}
\newtheorem{remark}[theorem]{Remark}

\def\<{\langle}
\def\>{\rangle}
\def\d{{\rm d}}
\def\L{\mathcal{L}}
\def\div{{\rm div}}
\def\E{\mathbb{E}}
\def\N{\mathbb{N}}
\def\P{\mathbb{P}}
\def\R{\mathbb{R}}
\def\T{\mathbb{T}}
\def\Z{\mathbb{Z}}
\def\Leb{{\rm Leb}}
\def\eps{\varepsilon}

\begin{document}

%\maketitle
\makeatletter
%'@' is now a normal "letter" for TeX
\renewcommand\theequation{\thesection.\arabic{equation}}
\@addtoreset{equation}{section} \makeatother
%'@' is restored as a "non-letter" character for TeX

\title{Point vortex approximation for 2D Navier--Stokes equations driven by space-time white noise}

\author{Franco Flandoli\footnote{Email: franco.flandoli@sns.it. Scuola Normale Superiore of Pisa, Italy.} \ and
Dejun Luo\footnote{Email: luodj@amss.ac.cn. RCSDS, Academy of Mathematics and Systems Science, Chinese Academy of Sciences, Beijing 100190, China, and School of Mathematical Sciences, University of the Chinese Academy of Sciences, Beijing 100049, China. }}

\maketitle

\begin{abstract}
We show that the system of point vortices, perturbed by a certain transport type noise, converges weakly to the vorticity form of 2D Navier--Stokes equations driven by the space-time white noise.
\end{abstract}



\section{Introduction}

The purpose of this paper is to show that a particle system of stochastic point vortices converges, as the number of particles goes to infinity, to the vorticity form of the Navier--Stokes equations driven by the space-time white noise:
  \begin{equation}\label{vorticity-NSE}
  \d \omega + u\cdot \nabla \omega  \,\d t = \Delta \omega \,\d t+ \sqrt{2}\, \nabla^\perp \cdot \d W, \quad \omega_0\stackrel{d}{\sim} \mbox{white noise on } \T^2.
  \end{equation}
Here $u= (u_1, u_2)$ is a divergence free vector field on the torus $\T^2= \R^2/\Z^2$ and $\omega = \nabla^\perp\cdot u = \partial_2 u_1- \partial_1 u_2$ is the vorticity. The equation \eqref{vorticity-NSE} in velocity-pressure variables reads as
  \begin{equation*}
  \aligned
  \d u+\left(  u\cdot\nabla u+\nabla p\right)  \d t  & = \Delta u\, \d t+ \sqrt{2}\, \d W,\\
  \operatorname{div}u  & =0,
  \endaligned
  \end{equation*}
which has been studied intensively in the last two decades, see for instance \cite{AC, DaPD, Debussche, AF, Stannat, AF2, Sauer, ZhuZhu} among others. This equation has an invariant measure given by some Gaussian measure $\mu$ which is supported by any Sobolev or Besov spaces of negative order. It was shown in  \cite[Theorem 5.2]{DaPD} that, for $\mu$-a.s. starting points in some Besov space, the above equation has a unique solution with continuous paths; moreover, if the initial data is a random variable with distribution $\mu$, then the solution is a stationary process.

To motivate our study we begin by considering the vorticity form of the 2D Euler equation:
  $$\partial_t \omega + u\cdot \nabla \omega =0, \quad \omega|_{t=0} =\omega_0.$$
This is a nonlinear transport equation in which $u$ is expressed by $\omega$ via the Biot--Savart law:
  $$u(x) = (K\ast \omega)(x) = \<\omega, K(x-\cdot)\>,$$
where $K$ is the Biot--Savart kernel on $\T^2$. We refer the readers to \cite[Introduction]{F1} for a list of well posedness results on this equation. In particular, we are interested in the case when $\omega_0$ has the form $\omega^N_0(\d x) = \frac1{\sqrt N}\sum_{i=1}^N \xi_i \delta_{X^i_0}(\d x)$, where $\xi_i\in \R$  and $X^i_0 \in \T^2$ are some distinct points. According to \cite[Section 4.4]{MP}, the above equation can be interpreted as the finite dimensional dynamics on $(\T^2)^N$:
  \begin{equation}\label{point-vortices}
  \frac{\d X^{i,N}_t}{\d t} = \frac1{\sqrt N} \sum_{j=1,j\neq i}^N \xi_j K\big(X^{i,N}_t -X^{j,N}_t\big)
  \end{equation}
with initial condition $X^{i,N}_0 = X^i_0,\, i=1,\cdots, N$. This system is not necessarily well posed: an explicit example was given in \cite[Section 4.2]{MP} which shows that three different vortex points starting from certain positions collapse to one point in finite time. Nevertheless, the above system of equations admits a unique solution for $\big( {\rm Leb}_{\T^2}^{\otimes N } \big)$-a.e. starting point in $(\T^2)^N$.

Based on the above result, the first author of the current paper considered the system \eqref{point-vortices} with random initial data $\omega^N_0$ which converges weakly to the white noise on $\T^2$ (see \cite[Section 3.2]{F1} or the next section for the precise meaning).  Denote by $\omega^N_t = \frac1{\sqrt N}\sum_{i=1}^N \xi_i \delta_{X^{i,N}_t}$. He proved in \cite[Theorem 24]{F1} that the family $\{\omega^N_\cdot \}$ has a subsequence which converges weakly to some $\omega_\cdot$ with continuous paths in $H^{-1-}(\T^2) = \cap_{s>0} H^{-1-s}(\T^2)$, such that $\omega_t$ is a white noise on $\T^2$ for all $t>0$. Furthermore, the process $\omega_\cdot$ solves the weak vorticity formulation of 2D Euler equations. We refer to \cite[Theorem 25]{F1} for more general results and to \cite{FL-1} for extensions to stochastic settings. On the other hand, we considered in the recent paper \cite{FL-3} the following stochastic 2D Euler equation
  $$\d\omega +u\cdot\nabla\omega\, \d t=2\,\varepsilon_{N}\sum_{0<|k|\leq N} e_{k}\frac{k^{\perp}}{|k|^{2}}\cdot\nabla\omega\circ \d W^{k}, $$
where $k$ runs over $\Z^2$, $\varepsilon_{N}= \big(\sum_{0<|k|\leq N} \frac1{|k|^2} \big)^{-1/2} \sim (\log N)^{-1/2}$, $\{e_{k} \}$ is the orthonormal basis of sine and cosine functions (see \eqref{ONB}) and $\{W^{k}\}$ are independent Brownian motions. It was shown that this model, hyperbolic in nature, converges to the parabolic equation \eqref{vorticity-NSE} above.

Motivated by the above discussions, we shall study in the current paper the stochastic point vortex dynamics
  \begin{equation*}
  \d X^{i,N}_t= \frac1{\sqrt N} \sum_{j=1,j\neq i}^N \xi_j K\big(X^{i,N}_t -X^{j,N}_t\big) \,\d t + 2\,\varepsilon_{N}\sum_{0<|k|\leq N}\frac{k^{\perp}}{|k|^{2}} e_k\big(X^{i,N}_t \big) \circ \d W^k_t.
  \end{equation*}
Assume the initial point vortices $\omega^N_0$ are random and converge weakly to the white noise on $\T^2$, we can prove that the processes $\omega^N_t = \frac1{\sqrt N}\sum_{i=1}^N \xi_i \delta_{X^{i,N}_t}$ converge weakly to the white noise solution of \eqref{vorticity-NSE}. The proof follows the general idea of \cite{FL-3} but we need some $L^2$-boundedness estimate on a sequence of functionals of $\omega^N_0$, which is done in the appendix.

\section{Convergence of the stochastic point vortex systems}\label{sec-convergence}

First, we introduce some notations. As in \cite[Section 3.2]{F1}, let $\{\xi_i\}_{i\in\N}$ be a family of i.i.d. $N(0,1)$ r.v.'s, and $\{X^i_0 \}_{i\in \N}$ is an i.i.d. sequence of $\T^2$-uniformly distributed r.v.'s; we assume the two families are independent. For every $N\in \N$, denote by
  $$\lambda_N^0= \big(N(0,1) \otimes \Leb_{\T^2}\big)^{\otimes N }$$
the law of the random vector $\big( (\xi_1, X^1_0), \ldots, (\xi_N, X^N_0)\big)$. Let us consider the measure-valued vorticity field
  \begin{equation*}
  \omega^N_0 = \frac1{\sqrt N} \sum_{i=1}^N \xi_i \delta_{X^i_0},
  \end{equation*}
which can be regarded as a r.v. taking values in the space $H^{-1-}(\T^2)= \cap_{s>0} H^{-1-s}(\T^2)$ with the law $\mu_N^0$, where $H^r(\T^2)\, (r\in \R)$ is the usual Sobolev space on $\T^2$. Denote by $\mathcal M(\T^2)$ the space of signed measures on $\T^2$ with finite variation, and
  $$\mathcal M_N(\T^2)= \big\{\mu \in \mathcal M(\T^2)\,| \,\exists \, X\subset \T^2 \mbox{ such that } \#(X)=N \mbox{ and } {\rm supp}(\mu) = X\big\}.$$
We can define the map $\mathcal T_N: (\R \times \T^2)^N \to \mathcal M_N(\T^2) \subset H^{-1-}(\T^2)$ as
  \begin{equation}\label{mapping}
  \big( (\xi_1, X^1_0), \ldots, (\xi_N, X^N_0)\big) \mapsto \omega^N_0= \frac1{\sqrt N} \sum_{i=1}^N \xi_i \delta_{X^i_0},
  \end{equation}
then it holds that
  $$\mu_N^0 = (\mathcal T_N)_\# \lambda_N^0 = \lambda_N^0\circ \mathcal T_N^{-1}.$$
It is proved in \cite[Proposition 21]{F1} that, as $N\to \infty$, $\omega_0^N$ converges in law to the white noise $\omega_{WN}$ on $\T^2$.

We denote by
  \begin{equation}\label{ONB}
  e_k(x)= \sqrt{2} \begin{cases}
  \cos(2\pi k\cdot x), & k\in \Z^2_+, \\
  \sin(2\pi k\cdot x), & k\in \Z^2_-,
  \end{cases} \quad x\in \T^2,
  \end{equation}
where $\Z^2_+ = \big\{k\in \Z^2_0: (k_1 >0) \mbox{ or } (k_1=0,\, k_2>0) \big\}$ and $\Z^2_- = -\Z^2_+$. Then $\{e_k: k\in \Z_0^2\}$ constitute a CONS of $L^2_0(\T^2)$, the space of square integrable functions with zero mean. Define
  \begin{equation}\label{vector-fields}
  \sigma_k(x)= \frac1{\sqrt{2} } \frac{k^\perp}{|k|^2} e_k(x), \quad k\in \Z^2_0,
  \end{equation}
with $k^\perp = (k_2,-k_1)$. For $N\geq 1$, define $\Lambda_N= \{k\in \Z_0^2: |k| \leq N\}$. Let $\{W^k_t\}_{k\in\Z_0^2}$ be a sequence of independent standard Brownian motions, which are independent of $\{\xi_i\}_{i\in\N}$ and $\{X^i_0 \}_{i\in \N}$. Consider the stochastic point vortex dynamics:  for $i=1,\cdots, N$,
  \begin{equation}\label{stoch-point-vortices-1}
  \d X^{i,N}_t= \frac1{\sqrt N} \sum_{j=1,j\neq i}^N \xi_j K\big(X^{i,N}_t -X^{j,N}_t\big)\,\d t + 2\sqrt{2}\, \eps_N \sum_{k\in \Lambda_N} \sigma_k\big(X^{i,N}_t \big) \circ\d W^k_t
  \end{equation}
with the initial condition $X^{i,N}_0= X^i_0$. Denote by
  $$\Delta_N= \big\{(x_1,\ldots, x_N)\in (\T^2)^N: \mbox{there are } i\neq j \mbox{ such that } x_i=x_j \big\}$$
the generalized diagonal of $(\T^2)^N$ and $\Delta_N^c =(\T^2)^N \setminus \Delta_N$. Moreover, for any $\phi\in C^\infty(\T^2)$, set
  $$H_\phi(x,y) = \frac12 K(x-y) \cdot (\nabla\phi(x) -\nabla\phi(y)), \quad x,y\in \T^2,$$
with the convention that $H_\phi(x,x) = 0$. It is well known that, for all $x\in \T^2 \setminus \{0\}$, $K(-x) =-K(x)$ and $|K(x)|\leq C/|x|$ for some constant $C>0$; thus $H_\phi$ is symmetric and
  \begin{equation}\label{non-linear-drift}
  \|H_\phi\|_\infty \leq C \|\nabla^2\phi\|_\infty.
  \end{equation}
We have the following result.

\begin{proposition}\label{prop-1}
For a.s. value of $\big( (\xi_1, X^1_0), \ldots, (\xi_N, X^N_0)\big)$, the process $\big( X^{1,N}_t, \ldots, X^{N,N}_t\big)$ is well defined in $\Delta_N^c$ for all $t\geq 0$, and the associated random measure-valued vorticity
  $$\omega^N_t = \frac1{\sqrt N} \sum_{i=1}^N \xi_i \delta_{X^{i,N}_t}$$
satisfies the equation below: for all $\phi \in C^\infty(\T^2)$,
  \begin{equation}\label{prop-stationary.1}
  \aligned
  \big\< \omega^N_t, \phi\big\>=&\, \big\< \omega^N_0, \phi\big\> +\int_0^t \big\<\omega^N_s\otimes \omega^N_s, H_\phi \big\>\, \d s + \int_0^t \<\omega^N_s, \Delta \phi \>\,\d s\\
  &\, + 2\sqrt{2}\, \eps_N \sum_{k\in \Lambda_N} \int_0^t \big\< \omega^N_s,\sigma_k \cdot \nabla \phi \big\>\,\d W^k_s.
  \endaligned
  \end{equation}
The stochastic process $\omega^N_t$ is stationary in time, with the law $\mu_N^0$ at any time $t\geq 0$.
\end{proposition}

\begin{proof}
The assertions are the same as \cite[Proposition 2.3]{FL-1}; the only difference is that here we can compute explicitly the second order derivative to get the Laplacian in the equation \eqref{prop-stationary.1}. Indeed, (2.5) in \cite{FL-1} becomes
  $$\aligned
  \big\< \omega^N_t, \phi\big\>=&\, \big\< \omega^N_0, \phi\big\> + \int_0^t \big\<\omega^N_s\otimes \omega^N_s, H_\phi \big\>\, \d s + 2\sqrt{2}\, \eps_N \sum_{k\in \Lambda_N} \int_0^t \big\< \omega^N_s,\sigma_k \cdot \nabla \phi \big\>\,\d W^k_s\\
  &\, + 4\eps_N^2 \sum_{k\in \Lambda_N} \int_0^t \big\< \omega^N_s, \sigma_k \cdot \nabla(\sigma_k \cdot \nabla \phi) \big\>\,\d s.
  \endaligned $$
We have $\sigma_k \cdot \nabla(\sigma_k \cdot \nabla \phi) =  {\rm Tr}\big[(\sigma_k\otimes \sigma_k)\nabla^2\phi\big]$ since $\sigma_k \cdot \nabla \sigma_k \equiv 0$ for any $k\in \Z_0^2$. The equation \eqref{prop-stationary.1} is a consequence of the following equality:
  \begin{equation}\label{prop-1.1}
  \sum_{k\in \Lambda_N} \sigma_k\otimes \sigma_k= \frac14 \eps_N^{-2} I_2,
  \end{equation}
where $I_2$ is the $(2\times 2)$-unit matrix. This identity was proved in \cite[Lemma 2.6]{FL-3}; we present the proof here for the reader's convenience. We have
  $$\aligned
  S_N(x) &:= \sum_{k\in \Lambda_N} \sigma_k(x)\otimes \sigma_k(x) = \sum_{k\in \Lambda_N \cap \Z^2_+} \frac{k^\perp \otimes k^\perp} {|k|^4} \big[\cos^2(2\pi k\cdot x) + \sin^2(2\pi k\cdot x)\big] \\
  &= \sum_{k\in \Lambda_N \cap \Z^2_+} \frac{1} {|k|^4} \begin{pmatrix}
  k_2^2 &  -k_1 k_2 \\  -k_1 k_2 & k_1^2
  \end{pmatrix}
  = \frac12 \sum_{k\in \Lambda_N} \frac{1} {|k|^4} \begin{pmatrix}
  k_2^2 &  -k_1 k_2 \\  -k_1 k_2 & k_1^2
  \end{pmatrix}.
  \endaligned $$
So $S_N$ is independent of $x$. First, we have
  $$S_N^{1,2}= - \frac12 \sum_{k\in \Lambda_N} \frac{k_1 k_2} {|k|^4} =0$$
since we can sum the four terms involving $(k_1, k_2),\, (-k_1, k_2),\, (k_1, -k_2),\, (-k_1, -k_2)$ at one time. Next,
  $$S_N^{1,1}=  \frac12 \sum_{k\in \Lambda_N} \frac{k_2^2} {|k|^4} = \frac12 \sum_{k\in \Lambda_N} \frac{k_1^2} {|k|^4} = S_N^{2,2}$$
since the points $(k_1, k_2)$ and $(k_2, k_1)$ appear in pair. Therefore,
  $$ S_N^{1,1}= S_N^{2,2} = \frac14 \sum_{k\in \Lambda_N} \frac{k_1^2 + k_2^2 } {|k|^4} = \frac14 \sum_{k\in \Lambda_N} \frac{1} {|k|^2} = \frac14 \varepsilon_N^{-2}.$$
Hence we obtain \eqref{prop-1.1}.
\end{proof}

Let $Q^N$ be the law of $\omega^N_\cdot$ on $\mathcal X= C\big([0,T], H^{-1-}(\T^2) \big),\, N\geq 1$. We want to show that the family $\{Q^N\}_{N\geq 1}$ is tight in $\mathcal X$, for which we need the following integrability properties of $\omega^N_t$ that are proved in \cite[Lemma 23]{F1} (except the second estimate which can be proved similarly to the first one).

\begin{lemma}\label{2-lem-integrability}
Assume $f:\T^2\times \T^2\to \R$ and $g:\T^2 \to \R$ are bounded and measurable, and $f$ is symmetric. Then, for every $p\geq 1$ and $\delta>0$, there are constants $C_p, C_{p,\delta}>0$ such that for all $N\geq 1$ and $t\in [0,T]$,
  $$\E\big[ \big| \big\<\omega^N_t \otimes \omega^N_t, f \big\> \big|^p \big] \leq C_p \|f\|_\infty^p,\quad \E\big[ \big| \big\<\omega^N_t , g \big\> \big|^p \big]\leq C_p \|g\|_\infty^p, \quad \E\big[ \big\|\omega^N_t \big\|_{H^{-1-\delta}}^p \big] \leq C_{p,\delta}.$$
Moreover,
  $$\E\big[ \big\<\omega^N_t \otimes \omega^N_t, f \big\>^2 \big]= \frac3N \! \int\! f^2(x,x)\,\d x +\frac{N-1}N \bigg[\int f(x,x)\,\d x \bigg]^2 + \frac{2(N-1)}N \! \int\!\int f^2(x,y)\,\d x\d y.$$
\end{lemma}

With these estimates in hand, we can follow the arguments at the beginning of \cite[Section 3]{FL-1} to show the tightness of $\{Q^N\}_{N\geq 1}$ in $\mathcal X$. To this end, we need to prove that $\{Q^N\}_{N\geq 1}$ is bounded in probability in $W^{1/3,4} \big(0,T; H^{-\kappa}(\T^2) \big)$ for some $\kappa>5$, and in $L^{p_0}\big(0,T; H^{-1-\delta}(\T^2) \big)$ for any $p_0>0$ and $\delta>0$.

First, by Lemma \ref{2-lem-integrability}, for all $N\in \N$,
  \begin{equation}\label{estimate-1}
  \E\bigg[ \int_0^T \big\| \omega^N_t\big\|_{H^{-1-\delta}}^{p_0} \,\d t\bigg] = \int_0^T \E \big[ \big\| \omega^N_t\big\|_{H^{-1-\delta}}^{p_0} \big] \,\d t \leq C_{p_0,\delta} T.
  \end{equation}
This implies the boundedness in probability of $\{Q^N\}_{N\geq 1}$ in $L^{p_0}\big(0,T; H^{-1-\delta}(\T^2) \big)$ for any $p_0>0$ and $\delta>0$.

Next, to show that $\{Q^N\}_{N\geq 1}$ is bounded in probability in $W^{1/3,4} \big(0,T; H^{-\kappa}(\T^2) \big)$ with $\kappa>5$, it suffices to prove
  $$\sup_{N\geq 1} \E\bigg[ \int_0^T \big\| \omega^N_t\big\|_{H^{-\kappa}}^4 \,\d t + \int_0^T \!\!\int_0^T \frac{\big\| \omega^N_t -\omega^N_s\big\|_{H^{-\kappa}}^4}{|t-s|^{7/3}} \,\d t\d s \bigg] <\infty. $$
The expectation of the first part is finite by the estimate \eqref{estimate-1}, thus we focus on the second part. We need the following result whose proof looks very similar to \cite[Lemma 2.5]{FL-3}. The difference between them is that here the processes $\omega^N_t$ are random point vortices, while the processes in \cite[Lemma 2.5]{FL-3} have white noise as marginal distribution.

\begin{lemma}\label{lem-estimate}
There exists $C>0$ such that for any $N\geq 1$ and $\phi\in C^\infty(\T^2)$, we have
  $$\E\big[ \big\<\omega^N_t- \omega^N_s, \phi \big\>^4 \big]\leq C (t-s)^2\big( \|\nabla \phi\|_\infty^4 + \|\nabla^2 \phi\|_\infty^4 \big).$$
\end{lemma}

\begin{proof}
By \eqref{prop-stationary.1}, one has
  \begin{equation}\label{lem-estimate.1}
  \aligned
  \big\< \omega^N_t -\omega^N_s, \phi\big\>=&\, \int_s^t \big\<\omega^N_r\otimes \omega^N_r, H_\phi \big\>\, \d r + \int_s^t \<\omega^N_r, \Delta \phi \>\,\d r\\
  &\, + 2\sqrt{2}\, \eps_N \sum_{k\in \Lambda_N} \int_s^t \big\< \omega^N_r,\sigma_k \cdot \nabla \phi \big\>\,\d W^k_r.
  \endaligned
  \end{equation}
First, H\"older's inequality leads to
  \begin{equation}\label{lem-estimate-2}
  \aligned
  \E\bigg[\bigg(\int_s^t \big\<\omega^N_r\otimes \omega^N_r, H_\phi\big\>\, \d r\bigg)^{\! 4} \bigg]
  &\leq (t-s)^3\, \E\bigg[\int_s^t \big\<\omega^N_r\otimes \omega^N_r, H_\phi \big\>^4\, \d r\bigg]\\
  &\leq (t-s)^3 \int_s^t C\|\nabla^2 \phi\|_\infty^4 \,\d r = C (t-s)^4  \|\nabla^2 \phi\|_\infty^4,
  \endaligned
  \end{equation}
where in the second step we used Lemma \ref{2-lem-integrability} and \eqref{non-linear-drift}. In the same way,
  \begin{equation}\label{lem-estimate-3}
  \E\bigg[\bigg(\int_s^t \<\omega^N_r, \Delta \phi \>\,\d r \bigg)^{\! 4} \bigg] \leq (t-s)^3\, \E\bigg[\int_s^t \<\omega^N_r, \Delta \phi \>^4\, \d r\bigg] \leq C(t-s)^4  \|\Delta \phi \|_\infty^4.
  \end{equation}

Next, by Burkholder's inequality,
  \begin{equation*}
  \aligned
  \E\bigg[\bigg(\varepsilon_N \sum_{k\in \Lambda_N} \int_s^t \big\<\omega^N_r,\sigma_k \cdot \nabla \phi \big\>\,\d W^k_r\bigg)^{\! 4} \bigg]
  &\leq C\varepsilon_N^4 \E \bigg[\bigg(\int_s^t \sum_{k\in \Lambda_N} \big\<\omega^N_r,\sigma_k \cdot \nabla \phi\big\>^2\,\d r\bigg)^{\! 2} \bigg]\\
  &\leq C\varepsilon_N^4 (t-s) \int_s^t \E \bigg[\bigg(\sum_{k\in \Lambda_N} \big\<\omega^N_r,\sigma_k \cdot \nabla \phi \big\>^2 \bigg)^{\! 2} \bigg] \d r.
  \endaligned
  \end{equation*}
Cauchy's inequality and Lemma \ref{2-lem-integrability} imply that
  $$  \aligned
  \E \bigg[\bigg(\sum_{k\in \Lambda_N} \big\<\omega^N_r,\sigma_k \cdot \nabla \phi \big\>^2 \bigg)^{\! 2} \bigg]
  &= \sum_{k,l \in \Lambda_N} \E \big[ \big\<\omega^N_r,\sigma_k \cdot \nabla \phi \big\>^2 \big\<\omega^N_r,\sigma_l \cdot \nabla \phi \big\>^2 \big] \\
  &\leq \sum_{k,l \in \Lambda_N} \big[\E \big\<\omega^N_r,\sigma_k \cdot \nabla \phi\big\>^4\big]^{1/2} \big[\E \big\<\omega^N_r,\sigma_l \cdot \nabla \phi\big\>^4\big]^{1/2}\\
  &\leq C\bigg(\sum_{k\in \Lambda_N} \|\sigma_k \cdot \nabla \phi\|_\infty^2\bigg)^{\! 2} \leq \tilde C \|\nabla \phi\|_\infty^4 \bigg(\sum_{k\in \Lambda_N} \|\sigma_k\|_\infty^2\bigg)^{\! 2}.
  \endaligned $$
Note that
  $$\sum_{k\in \Lambda_N} \|\sigma_k\|_\infty^2 = \sum_{k\in \Lambda_N} \frac1{|k|^2} = \varepsilon_N^{-2},$$
hence,
  $$\E \bigg[\bigg(\sum_{k\in \Lambda_N} \big\<\omega^N_r,\sigma_k \cdot \nabla \phi \big\>^2 \bigg)^{\! 2} \bigg] \leq C \|\nabla \phi\|_\infty^4\, \varepsilon_N^{-4}.$$
This implies
  $$\E\bigg[\bigg(\varepsilon_N \sum_{k\in \Lambda_N} \int_s^t \big\<\omega^N_r,\sigma_k \cdot \nabla \phi \big\>\,\d W^k_r\bigg)^{\! 4} \bigg] \leq C(t-s)^2 \|\nabla \phi\|_\infty^4. $$
Combining this estimate with \eqref{lem-estimate.1}--\eqref{lem-estimate-3}, we obtain the desired result.
\end{proof}

Applying Lemma \ref{lem-estimate} with $\phi(x)= e_k(x)$ leads to
  $$\E\big[ \big| \big\<\omega^N_t- \omega^N_s, e_k \big\> \big|^4 \big] \leq C (t-s)^2 |k|^8, \quad k\in \Z^2_0 .$$
As a result, by Cauchy's inequality,
  $$\aligned
  \E \big( \big\|\omega^N_t- \omega^N_s \big\|_{H^{-\kappa}}^4 \big) &= \E\bigg[\bigg( \sum_k \big(1+|k|^2 \big)^{-\kappa} \big|\big\<\omega^N_t- \omega^N_s, e_k \big\> \big|^2 \bigg)^{\! 2} \bigg]\\
  &\leq \bigg(\sum_k \big(1+|k|^2 \big)^{-\kappa}\bigg) \sum_k \big(1+|k|^2 \big)^{-\kappa} \E \big[ \big|\big\<\omega^N_t- \omega^N_s, e_k \big\> \big|^4 \big]\\
  &\leq \tilde C (t-s)^2\sum_k \big(1+|k|^2 \big)^{-\kappa} |k|^8 \leq \hat C (t-s)^2,
  \endaligned$$
since $2\kappa -8 >2$ due to the choice of $\kappa$. Consequently,
  $$\E \bigg[\int_0^T\! \int_0^T \frac{ \big\|\omega^N_t- \omega^N_s \big\|_{H^{-\kappa}}^4} {|t-s|^{7/3}}\,\d t\d s\bigg] \leq \hat C \int_0^T\! \int_0^T \frac{|t-s|^2} {|t-s|^{7/3}}\,\d t\d s <\infty.$$
The proof of the boundedness in probability of $\big\{Q^N\big\}_{N\geq 1}$ in $W^{1/3,4}\big(0,T; H^{-\kappa}(\T^2) \big)$ is complete.

Combining this result with \eqref{estimate-1} and the discussions below Lemma \ref{2-lem-integrability}, we conclude that $\{Q^N\}_{N\geq 1}$ is tight in $\mathcal{X} = C\big([0,T], H^{-1-}(\T^2) \big)$.

Since we are dealing with the SDEs \eqref{prop-stationary.1}, we need to consider $Q^N$ together with the distribution of Brownian motions. Although we use only finitely many Brownian motions in \eqref{prop-stationary.1}, here we consider for simplicity the whole family $\big\{ (W^k_t)_{0\leq t\leq T}: k\in \Z_0^2 \big\}$. To this end, we assume $\R^{\Z_0^2}$ is endowed with the metric
  $$d_{\Z_0^2}(a,b)= \sum_{k\in \Z_0^2} \frac{|a_k-b_k| \wedge 1}{2^{|k|}}, \quad a,b \in \R^{\Z_0^2}.$$
Then $\big( \R^{\Z_0^2}, d_{\Z_0^2} \big)$ is separable and complete (see \cite[Example 1.2, p.9]{Billingsley}). The distance in $\mathcal Y:= C\big([0,T], \R^{\Z_0^2} \big)$ is given by
  $$d_{\mathcal Y}(w,\hat w) = \sup_{t\in [0,T]} d_{\Z_0^2}(w(t), \hat w(t)),\quad w, \hat w \in \mathcal Y,$$
which makes $\mathcal Y$ a Polish space. Denote by $\mathcal W$ the law on $\mathcal Y$ of the sequence of independent Brownian motions $\big\{ (W^k_t)_{0\leq t\leq T}: k\in \Z_0^2\big\}$.

To simplify the notations, we write $W_\cdot= (W_t)_{0\leq t\leq T}$ for the whole sequence of processes $\big\{ (W^k_t)_{0\leq t\leq T}: k\in \Z_0^2 \big\}$ in $\mathcal Y$. Denote by $P^N$ the joint law of $\big(\omega^N_\cdot, W_\cdot \big)$ on $\mathcal X \times \mathcal Y,\, N\geq 1$. Since the marginal laws $\big\{ Q^N \big\}_{N\in \N}$ and $\{\mathcal W\}$ are respectively tight on $\mathcal X$ and $\mathcal Y$, we conclude that $\big\{ P^N \big\}_{N\in \N}$ is tight on $\mathcal X \times \mathcal Y$. By Skorokhod's representation theorem, there exist a subsequence $\{N_i \}_{i\in \N}$ of integers, a probability space $\big(\hat \Theta, \hat{\mathcal F}, \hat \P \big)$ and stochastic processes $\big(\hat \omega^{N_i}_\cdot, \hat W^{N_i}_\cdot\big)$ on this space with the corresponding laws $P^{N_i}$, and converging $\hat\P$-a.s. in $\mathcal X\times \mathcal Y$ to a limit $\big(\hat\omega_\cdot, \hat W_\cdot \big)$. We are going to prove that $\hat\omega_\cdot$, or more precisely another closely related process, solves the vorticity form of the Navier--Stokes equation with a suitable cylindrical Brownian motion.

We want to identify the approximating processes on the new probability space as random point vortices. For this purpose, we follow the discussions above \cite[Lemma 28]{F1} and enlarge the probability space $\big(\hat \Theta, \hat{\mathcal F}, \hat \P \big)$, so that it contains certain independent r.v.'s we need. The rough idea is to apply a random permutation to an $(\R\times \T^2)^N$-valued r.v. which corresponds, via the mapping \eqref{mapping}, to a r.v. with values in $\mathcal M_N(\T^2)$, see the end of \textbf{Step 1} in the proof of \cite[Lemma 28]{F1} for more details.  Denote by $\big(\tilde \Theta, \tilde {\mathcal F}, \tilde \P \big)$ a probability space on which, for every $N\geq 1$, we define a uniformly distributed random permutation $\tilde s_N: \tilde \Theta \to \Sigma_N$, where $\Sigma_N$ is the permutation group of order $N$. Define the product probability space
  \begin{equation}\label{product-space}
  (\Theta, \mathcal F, \P ) =\big(\hat \Theta\times \tilde \Theta, \hat{\mathcal F}\otimes \tilde {\mathcal F}, \hat \P \otimes \tilde \P\big)
  \end{equation}
and the new processes
  $$\big(\omega^{N_i}, W^{N_i}\big) = \big(\hat \omega^{N_i}, \hat W^{N_i}\big)\circ \pi_1, \quad (\omega, W) = \big(\hat \omega, \hat W \big)\circ \pi_1, \quad  s_N = \tilde s_N \circ \pi_2,$$
where $\pi_1$ and $\pi_2$ are the projections on $\hat \Theta\times \tilde \Theta$. Here, we slightly abuse the notations by denoting the final probability spaces and processes like the original ones. In the sequel we always consider the processes on the new probability space.

First, by Proposition \ref{prop-1}, it is easy to show

\begin{lemma}\label{lem-absolute-continuity}
The new process $\omega_\cdot$ is stationary and for every $t\in [0,T]$, the law $\mu_t$ of $\omega_t$ on $H^{-1-}(\T^2)$ is the white noise measure $\mu$.
\end{lemma}

Similarly to \cite[Lemma 3.5]{FL-1}, we can identify the structure of $\omega^{N_i}_t$ as a sum of Dirac masses.

\begin{lemma}\label{lem-1}
The process $\omega^{N_i}_t$  on the new probability space can be represented in the form $\frac1{\sqrt {N_i}} \sum_{j=1}^{N_i} \xi_j \delta_{X^{j, N_i}_t}$, where
  \begin{equation}\label{lem-1.0}
  \big(\big(\xi_1, X^{1,N_i}_0\big), \ldots, \big(\xi_{N_i}, X^{N_i,N_i}_0\big)\big)
  \end{equation}
is a random vector with law $\lambda_{N_i}^0$ and $\big(X^{1,N_i}_t, \ldots, X^{N_i,N_i}_t\big)$ solves the stochastic system \eqref{stoch-point-vortices-1} with the initial condition $\big(X^{1,N_i}_0, \ldots, X^{N_i,N_i}_0 \big)$ and new Brownian motions $\big\{\big( W^{N_i,k}_t \big): k\in \Lambda_{N_i} \big\}$ defined above.
\end{lemma}

As a consequence (cf. Proposition \ref{prop-1}), for any $i\in \N$ and $\phi\in C^\infty(\T^2)$, the new process $\omega^{N_i}_\cdot$ satisfies $\P$-a.s., for all $t\in [0,T]$,
  \begin{equation}\label{eq-seq}
  \aligned \big\< \omega^{N_i}_t, \phi\big\> =&\, \big\< \omega^{N_i}_0, \phi\big\> + \int_0^t \big\<\omega^{N_i}_s \otimes \omega^{N_i}_s, H_\phi \big\>\, \d s +\int_0^t \big\< \omega^{N_i}_s, \Delta \phi \big\>\,\d s \\
  &\, +2\sqrt{2}\,\eps_{N_i} \sum_{k\in \Lambda_{N_i}} \int_0^t \big\< \omega^{N_i}_s, \sigma_k \cdot \nabla \phi\big\>\,\d  W^{N_i,k}_s.
  \endaligned
  \end{equation}

\begin{remark}
Using the a.s. convergence of $\omega^{N_i}$ to $\omega$ in $C\big([0,T], H^{-1-}(\T^2)\big)$, we can show that the quantities in the first line of \eqref{eq-seq} converge respectively in $L^2(\Theta, \P )$ to
  $$\<\omega_t, \phi\>,\quad \<\omega_0, \phi\> ,\quad \int_0^t \big\<\omega_r\otimes \omega_r, H_\phi\big\>\, \d r,\quad \int_0^t \< \omega_r, \Delta \phi \>\,\d r,$$
see \cite[Proposition 3.6]{FL-1} for details. However, the term involving stochastic integrals does not converge strongly to some limit. Therefore, we can only seek for a weaker form of convergence.
\end{remark}

Before proceeding further, we introduce some notations. By $\Lambda\Subset \Z_0^2$ we mean that $\Lambda$ is a finite set. Let $\Pi_\Lambda: H^{-1-}(\T^2) \to \text{span}\{e_k: k\in \Lambda\}$ be the projection operator: $\Pi_\Lambda\omega = \sum_{l\in \Lambda} \<\omega, e_l\> e_l$. We shall use the family of cylindrical functions below:
  $$\mathcal{FC}_b^2 = \big\{F(\omega)= f(\<\omega, e_l\>; l\in \Lambda) \mbox{ for some } \Lambda\Subset \Z_0^2 \mbox{ and } f\in C_b^2\big(\R^\Lambda\big)\big\},$$
where $\R^\Lambda$ is the $(\#\Lambda)$-dimensional Euclidean space. To simplify the notations, sometimes we write the cylindrical functions as $F= f\circ \Pi_\Lambda$, and for $l,m\in \Lambda$, $f_l(\omega) = (\partial_l f)(\Pi_\Lambda \omega)$ and $f_{l,m}(\omega) = (\partial_l\partial_m f)(\Pi_\Lambda \omega)$. Denote by $\L_\infty$ the generator of the equation \eqref{vorticity-NSE}: for any cylindrical function $F= f \circ \Pi_\Lambda$ with $\Lambda \Subset \Z_0^2$,
  \begin{equation}\label{generator}
  \L_\infty F= 4 \pi^2 \sum_{l\in \Lambda} |l|^2 \big[f_{l,l}(\omega) - f_l(\omega) \< \omega, e_l \>\big] - \< u(\omega)\cdot \nabla\omega, D F\>,
  \end{equation}
where the drift part
  $$\< u(\omega)\cdot \nabla\omega, D F\>= - \sum_{l\in \Lambda} f_l(\omega) \big\<\omega\otimes \omega, H_{e_l}\big\>.$$
Finally we introduce the notation
  \begin{equation}\label{coefficients}
  C_{k,l} = \frac{k^\perp \cdot l}{|k|^2} ,\quad k,l \in\Z_0^2.
  \end{equation}
We have the following useful identity (cf. \cite[Lemma 3.4]{FL-2} for the proof):
  \begin{equation}\label{useful-identity}
  \sum_{k\in \Lambda_N} C_{k,l}^2 =\frac12 \eps_N^{-2} |l|^2
  \end{equation}

Now we prove that the limit $\omega$ is a martingale solution of the operator $\L_\infty$.

\begin{proposition}\label{approx-prop-1}
For any $F\in \mathcal{FC}_b^2$,
  \begin{equation}\label{approx-prop-1.1}
  M^F_t:= F(\omega_t ) -F(\omega_0 ) - \int_0^t \L_\infty F(\omega_s ) \,\d s
  \end{equation}
is an ${\mathcal F}_t = \sigma(\omega_s: s\leq t)$-martingale.
\end{proposition}

\begin{proof}
The proof below is analogous to that of \cite[Proposition 2.9]{FL-3}, but the processes $\tilde\omega^{N_i}_t$ involved there are processes of white noises on $\T^2$, while here $\omega^{N_i}_t$ are random point vortices. Recall the CONS defined in \eqref{ONB}. Taking $\phi= e_l$ in \eqref{eq-seq} for some $l\in \Z_0^2$, we have
  \begin{equation}\label{eq-seq-1}
  \aligned
  \d \big\< \omega^{N_i}_t, e_l\big\>&= \big\<\omega^{N_i}_t\otimes \omega^{N_i}_t, H_{e_l}\big\>\, \d t -4\pi^2 |l|^2 \big\< \omega^{N_i}_t, e_l \big\>\,\d t \\
  &\hskip13pt +2 \sqrt{2}\, \varepsilon_{N_i} \sum_{k\in \Lambda_{N_i}} \big\< \omega^{N_i}_t,\sigma_k \cdot \nabla e_l \big\>\,\d W^{{N_i}, k}_t.
  \endaligned
  \end{equation}
Therefore, for $l,m \in \Z_0^2$,
  \begin{equation*}
  \d \big\< \omega^{N_i}_t, e_l\big\> \cdot \d \big\< \omega^{N_i}_t, e_m\big\>= 8\varepsilon_{N_i}^2 \sum_{k\in \Lambda_{N_i}} \big\< \omega^{N_i}_t,\sigma_k \cdot \nabla e_l \big\> \big\< \omega^{N_i}_t,\sigma_k \cdot \nabla e_m \big\>\,\d t.
  \end{equation*}
It is easy to show that $\sigma_k \cdot \nabla e_l = \sqrt{2} \pi C_{k,l} e_k e_{-l}$; hence
  $$\aligned
  \big\< \omega^{N_i}_t,\sigma_k \cdot \nabla e_l \big\> \big\< \omega^{N_i}_t,\sigma_k \cdot \nabla e_m \big\> &= 2\pi^2 C_{k,l}C_{k,m} \big\< \omega^{N_i}_t, e_k e_{-l}\big\> \big\< \omega^{N_i}_t,e_k e_{-m} \big\> \\
  &= 2\pi^2 C_{k,l}C_{k,m} \Big[\big\< \omega^{N_i}_t, e_k e_{-l}\big\> \big\< \omega^{N_i}_t,e_k e_{-m} \big\> -\delta_{l,m} \Big]\\
  &\hskip13pt + 2\pi^2 \delta_{l,m} C_{k,l}^2.
  \endaligned$$
As a result,
  \begin{equation*}
  \aligned
  \d \big\< \omega^{N_i}_t, e_l\big\> \cdot \d \big\< \omega^{N_i}_t, e_m\big\> &= 16\pi^2 \varepsilon_{N_i}^2 \sum_{k\in \Lambda_{N_i}} C_{k,l}C_{k,m} \Big[\big\< \omega^{N_i}_t, e_k e_{-l}\big\> \big\< \omega^{N_i}_t,e_k e_{-m} \big\> -\delta_{l,m} \Big] \d t \\
  &\hskip13pt + 8\pi^2 \delta_{l,m} |l|^2 \,\d t,
  \endaligned
  \end{equation*}
where in the last step we have used \eqref{useful-identity}. To simplify the notations, we denote by
  $$R_{l,m}\big(\omega^{N_i}_t\big) = 8\pi^2 \sum_{k\in \Lambda_{N_i}} C_{k,l}C_{k,m} \Big[\big\< \omega^{N_i}_t, e_k e_{-l}\big\> \big\< \omega^{N_i}_t,e_k e_{-m} \big\> -\delta_{l,m} \Big].$$
Recall that $\omega^{N_i}_t$ has the law $\mu_{N_i}^0$ for any $t\in [0,T]$, thus $R_{l,m}\big(\omega^{N_i}_t\big)$ is bounded in $L^2\big([0,T]\times \Theta\big)$ by Proposition \ref{prop-square-bound} in the appendix. Finally, we get
  \begin{equation}\label{eq-seq-2}
  \d \big\< \omega^{N_i}_t, e_l\big\> \cdot \d \big\< \omega^{N_i}_t, e_m\big\> = 2 \varepsilon_{N_i}^2 R_{l,m}\big(\omega^{N_i}_t\big) \,\d t + 8\pi^2 \delta_{l,m} |l|^2 \,\d t.
  \end{equation}

By the It\^o formula and \eqref{eq-seq-1}, \eqref{eq-seq-2},
  \begin{equation*}
  \aligned
  \d F\big( \omega^{N_i}_t \big)= &\ \d f\big(\big\< \omega^{N_i}_t, e_l\big\>; l\in \Lambda\big)\\
  =&\ \sum_{l\in \Lambda} f_l\big( \omega^{N_i}_t \big) \Big[\big\< \omega^{N_i}_t\otimes \omega^{N_i}_t, H_{e_l}\big\> -4\pi^2 |l|^2 \big\< \omega^{N_i}_t, e_l \big\>\Big]\,\d t \\
  &\ + 2 \sqrt{2}\, \varepsilon_{N_i} \sum_{l\in \Lambda} f_l\big( \omega^{N_i}_t \big) \sum_{k\in \Lambda_{N_i}} \big\< \omega^{N_i}_t,\sigma_k \cdot \nabla e_l \big\>\,\d W^{{N_i}, k}_t \\
  &\ + \sum_{l, m\in \Lambda} f_{l,m}\big( \omega^{N_i}_t \big) \big[\varepsilon_{N_i}^2 R_{l,m}\big( \omega^{N_i}_t\big) + 4 \pi^2 \delta_{l,m} |l|^2 \big] \,\d t.
  \endaligned
  \end{equation*}
Recalling the operator $\L_\infty$ defined in \eqref{generator}, the above formula can be rewritten as
  \begin{equation}\label{Ito-formula}
  \d F\big( \omega^{N_i}_t \big) = \L_\infty F\big( \omega^{N_i}_t \big) \,\d t + \varepsilon_{N_i}^2 \zeta^{N_i}_t \,\d t + \d M^{N_i}_t,
  \end{equation}
where, by Proposition \ref{prop-square-bound},
  $$\zeta^{N_i}_t = \sum_{l, m\in \Lambda} f_{l,m}\big( \omega^{N_i}_t \big) R_{l,m}\big( \omega^{N_i}_t\big)$$
is bounded in $L^2\big([0,T]\times \Theta \big)$ since $\{f_{l,m} \}_{l, m\in \Lambda}$ are bounded, and the martingale part
  $$\d M^{N_i}_t= 2 \sqrt{2}\, \varepsilon_{N_i} \sum_{l\in \Lambda} f_l\big( \omega^{N_i}_t \big) \sum_{k\in \Lambda_{N_i}} \big\< \omega^{N_i}_t,\sigma_k \cdot \nabla e_l \big\>\,\d W^{{N_i}, k}_t. $$
Note that $M^{N_i}_t$ is a martingale w.r.t. the filtration
  $${\mathcal F}^{N_i}_t = \sigma\big( \omega^{N_i}_s, W^{N_i}_s: s\leq t\big),$$
where we denote by $W^{N_i}_s = \big\{ W^{N_i,k}_s \big\}_{k\in \Z_0^2}$.

Next, we show that the formula \eqref{Ito-formula} converges as $i\to \infty$ in a suitable sense. To this end, we follow the argument of \cite[p. 232]{DaPZ}. Fix any $0<s <t\leq T$. Take a real valued, bounded and continuous function $\varphi: C\big([0,s], H^{-1-} \times \R^{\Z_0^2} \big)\to \R$. By \eqref{Ito-formula}, we have
  $$\E\bigg[\bigg( F\big( \omega^{N_i}_t \big) -F\big( \omega^{N_i}_s \big) - \int_s^t \L_\infty F\big( \omega^{N_i}_r \big) \,\d r - \varepsilon_{N_i}^2 \int_s^t \zeta^{N_i}_r \,\d r \bigg) \varphi\big( \omega^{N_i}_\cdot, W^{N_i}_\cdot \big)\bigg] =0.$$
Since $F\in \mathcal{FC}_b^2$ and $\omega^{N_i}_t$ has the law $\mu_{N_i}^0$ for all $t\in [0,T]$, by Lemma \ref{2-lem-integrability}, all the terms in the round bracket are square integrable. Recall that, $\P$-a.s., $\big( \omega^{N_i}_\cdot, W^{{N_i}}_\cdot \big)$ converges to $\big( \omega_\cdot, W_\cdot \big)$ in $C\big([0,T], H^{-1-} \times \R^{\Z_0^2} \big)$. Repeating the treatment of the term $I^{N_k}_3$ in the proof of \cite[Proposition 3.6]{FL-1}, we can show the convergence of the term involving the nonlinear part in $\L_\infty F$; the other terms are simple. Thus, letting $i\to \infty$ in the above equality yields
  $$\E\bigg[\bigg( F(\omega_t ) -F(\omega_s ) - \int_s^t \L_\infty F(\omega_r ) \,\d r \bigg) \varphi\big(\omega_\cdot, W_\cdot \big)\bigg] =0.$$
The arbitrariness of $0<s<t$ and $\varphi: C\big([0,s], H^{-1-} \times \R^{\Z_0^2} \big)\to \R$ implies that $M^F_\cdot$ is a martingale with respect to the filtration ${\mathcal G}_t = \sigma\big( \omega_s, W_s: s\leq t\big),\, t\in [0,T]$. For any $0\leq s< t\leq T$, we have ${\mathcal F}_s\subset {\mathcal G}_s$, thus
  $$\E \big( M^F_t \big| {\mathcal F}_s \big)=  \E \Big[ \E \big( M^F_t \big| {\mathcal G}_s \big) \big| {\mathcal F}_s \Big] = \E \big[M^F_s \big| {\mathcal F}_s \big] = M^F_s ,$$
since $M^F_s$ is adapted to ${\mathcal F}_s $.
\end{proof}

At this stage, taking the cylinder functions $F(\omega) = \<\omega, e_l\>$ and $F(\omega) = \<\omega, e_l\> \<\omega, e_m\> \, (l, m\in \Z_0^2)$ and using L\'evy's characterization of Brownian motions, it is easy to show that (see \cite[Proposition 2.10]{FL-3} for details)

\begin{proposition}\label{prop-BM}
There exists a family of independent standard Brownian motions $\big\{ W^k_t: t\geq 0\big\}_{k\in \Z_0^2}$ such that $(\omega_\cdot, W_\cdot)$ solves \eqref{vorticity-NSE}, where $ W_t =\sum_{k\in \Z_0^2} W^{-k}_t e_k \frac{k^\perp}{|k|}$.
\end{proposition}

In the remaining part of this section, we follow the arguments at the end of \cite[Section 2]{FL-3}. We can rewrite \eqref{vorticity-NSE} in the velocity variable $\tilde u_\cdot= u(\tilde \omega_\cdot)$ as follows:
  \begin{equation}\label{NSE}
  \d\tilde u + b(\tilde u) \,\d t = \nu A \tilde u\,\d t + \sqrt{2\nu}\, \d \tilde W.
  \end{equation}
Here, $b(u) = \mathcal P \div(u\otimes u)$ and $Au = \mathcal P \Delta u$, in which $\mathcal P$ is the orthogonal projection onto the space of divergence free vector fields on $\T^2$. It is clear that $\tilde u$ has trajectories in $C\big( [0,T], H^{-}(\T^2)\big)$, that is, in $C\big( [0,T], H^{-\delta}(\T^2)\big)$ for any $\delta>0$. As mentioned at the beginning of this paper, the above equation has been studied intensively in the last two decades. We deduce from Lemma \ref{lem-absolute-continuity} and Proposition \ref{prop-BM} that the process $\tilde u$ is a stationary solution to \eqref{NSE} in the sense of \cite[Definition 4.1]{DaPD}. Let us remark that this definition is based only on the Sobolev regularity of $\tilde u\in C\big( [0,T], H^{-}(\T^2)\big)$; the definition of the nonlinear part $b(\tilde u)$ is based on the Galerkin approximation and coincides with our definition, as explained by \cite[Theorem A.12]{FL-3} in terms of the vorticity variable.

Similarly to the arguments in \cite[Section 3.5]{RZZ}, we can prove

\begin{proposition}\label{pathwise-uniqueness}
The uniqueness in law holds for stationary solutions to \eqref{NSE}.
\end{proposition}

\begin{proof}
By \cite[Theorem 3.14]{Kur}, it is sufficient to show that the pathwise uniqueness holds for stationary solutions of  \eqref{NSE}. Let $u_i \ (i=1,2)$ be two stationary solutions to the equation \eqref{NSE} in the sense of \cite[Definition 4.1]{DaPD}, which are defined on the same probability space $(\Theta, \mathcal F, \P)$, with the same initial data $u_1(0) = u_2(0) = u(0)$ ($\P$-a.s.) and the same cylindrical Brownian motion $W(t)$, $0\leq t \leq T$. Then, for $i=1,2$, $\P$-a.s.,
  $$u_i(t) = u(0) - \int_0^t b(u_i(s))\,\d s + \int_0^t A u_i(s)\,\d s + \sqrt{2}\,W(t),\quad 0\leq t\leq T.$$
These equations can be rewritten as
  $$u_i(t) = {\rm e}^{tA} u(0) - \int_0^t {\rm e}^{(t-s)A} b(u_i(s))\,\d s + \sqrt{2} \int_0^t {\rm e}^{(t-s)A}\, \d W(s). $$
We extend $W(\cdot)$ to be a two-sided cylindrical Brownian motion on $\R$ (possibly at the price of enlarging $(\Theta, \mathcal F, \P)$) and define
  $$Z(t)= \sqrt{2} \int_{-\infty}^t {\rm e}^{(t-s)A}\, \d W(s). $$
It is well known that $Z$ is a stationary process with paths in $C\big([0,T], B^\sigma_{p,\rho} \big)$ for any $\sigma<0,\, \rho\geq p\geq 2$ (cf. the last line on p.196 of \cite{DaPD}). Here, for any $s\in \R$, $B^s_{p,\rho}$ is the Besov space on $\T^2$. Note that
  $$ \sqrt{2} \int_0^t {\rm e}^{(t-s)A}\, \d W(s)= Z(t)- {\rm e}^{tA} Z(0),$$
we arrive at
  \begin{equation}\label{pathwise-uniqueness.1}
  u_i(t) -Z(t)= {\rm e}^{tA} (u(0) -Z(0)) - \int_0^t {\rm e}^{(t-s)A} b(u_i(s))\,\d s, \quad i=1,2.
  \end{equation}

As in \cite[Theorem 5.2, p.196]{DaPD}, let $\alpha, \beta, p, \rho, \sigma$ be such that
  $$\frac2p >\alpha> -\sigma>0,\, \rho=p \geq 2,\, \beta\geq 1, \, -\frac12 + \frac1p < \frac\alpha 2 -\frac1\beta <\frac\sigma2.$$
Using these parameters, we define the following space
  $$\mathcal E = L^\beta \big(0,T; B^\alpha_{p,\rho}\big) \cap C\big([0,T], B^\sigma_{p,\rho} \big).$$
Since for any $t\in [0,T]$, $u_i(t)$ is distributed as $\mathcal N(0, (-A)^{-1})= \otimes_{k\in \Z_0^2} N\big(0, 1/(4\pi^2 |k|^2) \big)$, one has $u_i(t) \in B^\sigma_{p,\rho},\ \P$-a.s. (see \cite[Proposition 3.1]{AF}). We also have $Z(0)\in B^\sigma_{p,\rho}$ ($\P$-a.s.), thus by \cite[Lemma 6.1]{DaPD}, we obtain that, $\P$-a.s., $[0,T]\ni t\mapsto {\rm e}^{tA} (u(0) -Z(0)) \in \mathcal E$. Next, for any $\gamma\geq 1$ and $\eps>0$, since
  $$\E\bigg(\int_0^T \|b(u_i(t))\|_{H^{-1-\eps}}^\gamma \,\d t\bigg) = \int_0^T \E\big( \|b(u_i(t))\|_{H^{-1-\eps}}^\gamma \big) \,\d t,$$
using estimates on the operator $b(\cdot)$ and the regularity provided by the Gaussian marginal of $u_i(\cdot)$, we can prove $b(u_i(\cdot)) \in L^\gamma\big(0,T; H^{-1-\eps}\big)$ ($\P$-a.s.), see the arguments on the top of p.197 in \cite{DaPD} for details. Therefore, \cite[Lemma 6.2]{DaPD} gives us that $\int_0^t {\rm e}^{(t-s)A} b(u_i(s))\,\d s \in \mathcal E$. Combining these discussions with the equations \eqref{pathwise-uniqueness.1}, we deduce that $u_i -Z \in \mathcal E$ ($\P$-a.s.) for $i=1,2$. By \cite[Theorem 5.2, p.196]{DaPD} (see in particular the arguments on p.200 after the proof), we obtain, $\P$-a.s., $u_1(t) = u_2(t)$ for all $t\in [0,T]$. Thus the pathwise uniqueness holds for stationary solutions to \eqref{NSE}.
\end{proof}

Recall that $\{Q^N\}_{N\geq 1}$ are the distributions of $\big(\omega^N_t \big)_{0\leq t\leq T}$. Now we can prove the main result of this paper.

\begin{theorem}\label{thm-convergence}
The whole sequence $\{Q^N\}_{N\geq 1}$ converges weakly to the distribution of solution to \eqref{vorticity-NSE}.
\end{theorem}

\begin{proof}
Proposition \ref{pathwise-uniqueness} implies that the stationary solutions to \eqref{vorticity-NSE} are unique in law, thus we deduce the assertion from the tightness of the family $\{Q^N\}_{N\geq 1}$.
\end{proof}

\section{Appendix}

Recall the expressions of $\omega^N_0$ in \eqref{mapping} and of $C_{k,l}$ in \eqref{coefficients}. In this part we prove the following technical result.

\begin{proposition}\label{prop-square-bound}
For any $l,m\in \Z^2_0$ fixed, the sequence of random variables
  $$R_{l,m}(\omega^N_0) = \sum_{k\in \Lambda_N} C_{k,l} C_{k,m} \big(\<\omega^N_0, e_ke_l\> \<\omega^N_0, e_ke_m\> -\delta_{l,m} \big) $$
is bounded in $L^2(\Theta, \mathcal F, \P)$.
\end{proposition}

The proof of the above assertion follows the idea of \cite[Appendix 6]{FL-2}, with some combinatorial flavor here. Since $l,m$ are fixed, we write $R_N$ instead of $R_{l,m}(\omega^N_0)$ for simplicity. We deal with the two cases $l\neq m$ and $l=m$ in the two subsections separately.

\subsection{Case 1: $l\neq m$}

The definition of  $\omega^N_0$ yields
  $$R_N = \frac1N \sum_{k\in \Lambda_N} C_{k,l} C_{k,m} \sum_{r,s=1}^N \xi_r \xi_s (e_k e_l)(X^r_0) (e_k e_m)(X^s_0),$$
therefore,
  $$\aligned
  R_N^2 = \frac1{N^2} \sum_{k, k'\in \Lambda_N} \sum_{r,s,r's'=1}^N &\, C_{k,l} C_{k,m} C_{k',l} C_{k',m}  \xi_r \xi_s \xi_{r'} \xi_{s'} \\
  &\, \times (e_k e_l)(X^r_0) (e_k e_m)(X^s_0) (e_{k'} e_l)(X^{r'}_0) (e_{k'} e_m)(X^{s'}_0).
  \endaligned$$
Recall that the two families $\{\xi_r \}_{r\geq 1}$ and $\{X^r_0 \}_{r\geq 1}$ are independent, and $\{\xi_r \}_{r\geq 1}$ is an i.i.d. sequence of $N(0,1)$ r.v.'s, while $\{X^r_0 \}_{r\geq 1}$ consists of i.i.d. $\T^2$-valued uniform r.v.'s. We have
  $$\aligned
  \E R_N^2 = \frac1{N^2} \sum_{k, k'\in \Lambda_N} \sum_{r,s,r',s'=1}^N &\, C_{k,l} C_{k,m} C_{k',l} C_{k',m}  \E(\xi_r \xi_s \xi_{r'} \xi_{s'}) \\
  &\, \times \E\big[(e_k e_l)(X^r_0) (e_k e_m)(X^s_0) (e_{k'} e_l)(X^{r'}_0) (e_{k'} e_m)(X^{s'}_0) \big]
  \endaligned$$
and by the Isserlis--Wick theorem,
  $$\aligned
  \E(\xi_r \xi_s \xi_{r'} \xi_{s'}) &= \E(\xi_r \xi_s ) \E(\xi_{r'} \xi_{s'}) + \E(\xi_r \xi_{r'}) \E(\xi_s \xi_{s'}) + \E(\xi_r \xi_{s'})\E(\xi_s \xi_{r'}) \\
  &= \delta_{r,s} \delta_{r',s'} + \delta_{r,r'} \delta_{s,s'} + \delta_{r,s'} \delta_{s,r'}.
  \endaligned $$
As a result, we can write
  \begin{equation}\label{app-1}
  \E R_N^2 = S_1 +S_2 + S_3.
  \end{equation}

\subsubsection{The quantity $S_1$}

We have
  $$S_1 = \frac1{N^2} \sum_{k, k'\in \Lambda_N} \sum_{r,r'=1}^N C_{k,l} C_{k,m} C_{k',l} C_{k',m} \E\big[(e_k^2 e_l e_m)(X^r_0) (e_{k'}^2 e_l e_m)(X^{r'}_0) \big].$$
Note that $X^r_0$ and $X^{r'}_0$ are independent if $r\neq r'$, hence
  \begin{equation}\label{S1-decomposition}
  \aligned
  S_1 =&\, \frac1{N^2} \sum_{k, k'\in \Lambda_N} \sum_{1\leq r\neq r'\leq N} C_{k,l} C_{k,m} C_{k',l} C_{k',m} \E\big[(e_k^2 e_l e_m)(X^r_0) \big] \E\big[(e_{k'}^2 e_l e_m)(X^{r'}_0) \big] \\
  &\, + \frac1{N^2} \sum_{k, k'\in \Lambda_N} \sum_{r=1}^N C_{k,l} C_{k,m} C_{k',l} C_{k',m} \E\big[(e_k^2 e_{k'}^2 e_l^2 e_m^2 )(X^r_0) \big] .
  \endaligned
  \end{equation}
We denote the two terms by $S_{1,1}$ and $S_{1,2}$, respectively.

First, since $X^r_0\ (r\in \N)$ is a uniformly distributed r.v. on the torus $\T^2$, we obtain
  $$\aligned
  S_{1,1} &= \frac1{N^2} \sum_{k, k'\in \Lambda_N} \sum_{1\leq r\neq r'\leq N} C_{k,l} C_{k,m} C_{k',l} C_{k',m} \int e_k^2 e_l e_m\,\d x \int e_{k'}^2 e_l e_m\,\d x \\
  &= \frac{N^2 -N}{N^2}\sum_{k, k'\in \Lambda_N} C_{k,l} C_{k,m} C_{k',l} C_{k',m} \int e_k^2 e_l e_m\,\d x \int e_{k'}^2 e_l e_m\,\d x \\
  &= \bigg(1- \frac1N \bigg) \bigg(\sum_{k\in \Lambda_N} C_{k,l} C_{k,m} \int e_k^2 e_l e_m\,\d x\bigg)^2.
  \endaligned$$
Note that $C_{-k,l} = -C_{k,l}$ and $e_k^2 + e_{-k}^2 \equiv 2$ for any $k\in \Z_0^2$, we have
  \begin{equation}\label{S1-0}
  \sum_{k\in \Lambda_N} C_{k,l} C_{k,m} e_k^2 = \sum_{k\in \Lambda_N \cap \Z^2_+} \big(C_{k,l} C_{k,m} e_k^2 + C_{-k,l} C_{-k,m} e_{-k}^2\big) = 2 \sum_{k\in \Lambda_N \cap \Z^2_+} C_{k,l} C_{k,m}
  \end{equation}
is a constant. This implies
  \begin{equation}\label{S1-1}
  S_{1,1} =0
  \end{equation}
since $\int e_l e_m\,\d x =0$ for $l\neq m$.

Regarding the term $S_{1,2}$, we have
  $$\aligned S_{1,2} &= \frac1{N^2} \sum_{k, k'\in \Lambda_N} \sum_{r=1}^N C_{k,l} C_{k,m} C_{k',l} C_{k',m} \int e_k^2 e_{k'}^2 e_l^2 e_m^2\,\d x \\
  &= \frac1N \sum_{k, k'\in \Lambda_N} C_{k,l} C_{k,m} C_{k',l} C_{k',m} \int e_k^2 e_{k'}^2 e_l^2 e_m^2\,\d x .
  \endaligned$$
As $|e_k(x)|\leq \sqrt{2}$ for all $x\in \T^2$ and $k\in \Z_0^2$, we deduce that
  $$|S_{1,2}| \leq \frac{16}N \sum_{k, k'\in \Lambda_N} \frac{|l|^2 |m|^2} {|k|^2 |k'|^2} = \frac{16}N |l|^2 |m|^2 \bigg(\sum_{k\in \Lambda_N} \frac{1} {|k|^2}\bigg)^2 \leq C(l,m) \frac{(\log N)^2}N .$$
Combining the above estimate with \eqref{S1-decomposition} and \eqref{S1-1}, we arrive at
  \begin{equation}\label{S1-2}
  |S_1|\leq C_1 \frac{(\log N)^2}N \quad \mbox{for all } N\geq 2.
  \end{equation}

\subsubsection{The quantity $S_2$}

We have
  $$S_2= \frac1{N^2} \sum_{k, k'\in \Lambda_N} \sum_{r,s=1}^N C_{k,l} C_{k,m} C_{k',l} C_{k',m} \E\big[(e_ke_{k'} e_l^2)(X^r_0) (e_ke_{k'} e_m^2)(X^s_0) \big]. $$
Similar to \eqref{S1-decomposition}, the above quantity can be decomposed as
  $$\aligned
  S_2 =&\, \frac1{N^2} \sum_{k, k'\in \Lambda_N} \sum_{1\leq r\neq s\leq N} C_{k,l} C_{k,m} C_{k',l} C_{k',m} \E\big[(e_ke_{k'} e_l^2)(X^r_0) \big] \E\big[(e_ke_{k'} e_m^2)(X^s_0) \big]\\
  &\, + \frac1{N^2} \sum_{k, k'\in \Lambda_N} \sum_{r=1}^N C_{k,l} C_{k,m} C_{k',l} C_{k',m} \E\big[(e_k^2 e_{k'}^2 e_l^2 e_m^2)(X^r_0) \big] ,
  \endaligned$$
which are denoted as $S_{2,1}$ and $S_{2,2}$. Note that
  $$|S_{2,2}| = |S_{1,2}| \leq C_1 \frac{(\log N)^2}N \quad \mbox{for all } N\geq 2.$$
Next, using the fact that $X^r_0$ is uniformly distributed on $\T^2$ and the Cauchy inequality,
  $$\aligned |S_{2,1}| &= \bigg| \bigg(1- \frac1N \bigg) \sum_{k, k'\in \Lambda_N} C_{k,l} C_{k,m} C_{k',l} C_{k',m} \int e_ke_{k'} e_l^2 \,\d x \int e_ke_{k'} e_m^2 \,\d x \bigg| \\
  & \leq \Bigg[ \sum_{k, k'\in \Lambda_N} C^2_{k,l} C^2_{k',l} \bigg( \int e_ke_{k'} e_l^2 \,\d x \bigg)^2 \Bigg]^{1/2} \Bigg[ \sum_{k, k'\in \Lambda_N} C^2_{k,m} C^2_{k',m} \bigg( \int e_ke_{k'} e_m^2 \,\d x \bigg)^2 \Bigg]^{1/2}.
  \endaligned$$
It suffices to estimate one of the two terms. Intuitively, the quantity
  \begin{equation}\label{S2-0}
  I_N:= \sum_{k, k'\in \Lambda_N} C^2_{k,l} C^2_{k',l} \bigg( \int e_ke_{k'} e_l^2 \,\d x \bigg)^2
  \end{equation}
is bounded as $N \to \infty$ due to the fact that the integral $\int e_k e_{k'} e_l^2 \,\d x \neq 0$ imposes a constraint on $k$ and $k'$, e.g. $k=k'$ or $2l =k+k'$. Such constraint reduces the degree of freedom of $k$ and $k'$, and implies
  $$I_N \leq C_l \sum_{k\in \Lambda_N} \frac1{|k|^4} \leq C_l \sum_{k\in \Z_0^2} \frac1{|k|^4} \quad \mbox{for all } N\geq 1.$$
We refer the readers to \cite[Section 6.1.2]{FL-2} for details.

To summarize, we obtain
  \begin{equation}\label{S2}
  |S_2 |\leq C_2\bigg(1+ \frac{(\log N)^2}N \bigg).
  \end{equation}

\subsubsection{The quantity $S_3$}

Similar computations as above lead to
  $$\aligned
  S_3 =&\, \frac1{N^2} \sum_{k, k'\in \Lambda_N} \sum_{r,s=1}^N C_{k,l} C_{k,m} C_{k',l} C_{k',m} \E\big[(e_k e_{k'} e_le_m)(X^r_0) (e_k e_{k'} e_l e_m)(X^s_0) \big]\\
  = &\, \frac1{N^2} \sum_{k, k'\in \Lambda_N} \sum_{1\leq r\neq s\leq N} C_{k,l} C_{k,m} C_{k',l} C_{k',m} \E\big[(e_k e_{k'} e_le_m)(X^r_0) \big] \E\big[(e_k e_{k'} e_l e_m)(X^s_0) \big] \\
  &\, + \frac1{N^2} \sum_{k, k'\in \Lambda_N} \sum_{r=1}^N C_{k,l} C_{k,m} C_{k',l} C_{k',m} \E\big[(e_k^2 e_{k'}^2 e_l^2 e_m^2)(X^r_0) \big] .
  \endaligned$$
Again, the last quantity is dominated by a constant multiple of $(\log N)^2/N$. The first one on the right hand side is equal to
  $$\bigg(1- \frac1N \bigg) \sum_{k, k'\in \Lambda_N} C_{k,l} C_{k,m} C_{k',l} C_{k',m} \bigg( \int e_k e_{k'} e_le_m \,\d x\bigg)^2,$$
which, due to the same reason as for the term \eqref{S2-0}, is bounded in $N$. Therefore, we still have
  $$|S_3| \leq C_3 \bigg(1+ \frac{(\log N)^2}N \bigg).$$

Combining the above inequality with \eqref{app-1}, \eqref{S1-2} and \eqref{S2}, we conclude the assertion in the first case $l\neq m$.

\subsection{Case 2: $l=m$}

In this case,
  $$R_N = \sum_{k\in \Lambda_N} C_{k,l}^2 \big(\<\omega^N_0, e_ke_l\>^2 -1 \big).$$
Consequently,
  \begin{equation}\label{case2}
  \E R_N^2 = \sum_{k, k'\in \Lambda_N} C_{k,l}^2 C_{k',l}^2\, \E \big(\<\omega^N_0, e_ke_l\>^2 \<\omega^N_0, e_{k'}e_l\>^2 -\<\omega^N_0, e_ke_l\>^2 - \<\omega^N_0, e_{k'}e_l\>^2 +1 \big).
  \end{equation}
By the definition of $\omega^N_0$,
  $$ \aligned \E \big(\<\omega^N_0, e_ke_l\>^2 \big) &= \frac1N \sum_{r,s=1}^N \E(\xi_r \xi_s) \E\big[(e_ke_l)(X^r_0) (e_ke_l)(X^s_0) \big] \\
  &= \frac1N \sum_{r=1}^N \E\big[(e_k^2 e_l^2)(X^r_0) \big] = \int e_k^2 e_l^2 \,\d x.
  \endaligned$$
As a result,
  \begin{equation}\label{case2-1}
  \sum_{k, k'\in \Lambda_N} C_{k,l}^2 C_{k',l}^2\, \E \big(\<\omega^N_0, e_ke_l\>^2 \big) = \bigg(\sum_{k'\in \Lambda_N} C_{k',l}^2\bigg) \sum_{k\in \Lambda_N} C_{k,l}^2 \int e_k^2 e_l^2 \,\d x.
  \end{equation}
Similar to \eqref{S1-0},
  \begin{equation}\label{case2-2}
  \sum_{k\in \Lambda_N} C_{k,l}^2 e_k^2 = 2 \sum_{k\in \Lambda_N\cap \Z^2_+} C_{k,l}^2 = \sum_{k\in \Lambda_N} C_{k,l}^2 =\frac12 \eps_N^{-2} |l|^2,
  \end{equation}
where the last step is due to \eqref{useful-identity}. Substituting this result into \eqref{case2-1} yields
  \begin{equation*}
  \sum_{k, k'\in \Lambda_N} C_{k,l}^2 C_{k',l}^2\, \E \big(\<\omega^N_0, e_ke_l\>^2 \big) = \frac14 \eps_N^{-4} |l|^4.
  \end{equation*}
Analogously,
  \begin{equation*}
  \sum_{k, k'\in \Lambda_N} C_{k,l}^2 C_{k',l}^2\, \E \big(\<\omega^N_0, e_{k'}e_l\>^2 \big) = \frac14 \eps_N^{-4} |l|^4.
  \end{equation*}
Combining these facts with \eqref{case2}, we obtain
  \begin{equation}\label{case2-3}
  \E R_N^2 = \sum_{k, k'\in \Lambda_N} C_{k,l}^2 C_{k',l}^2\, \E \big(\<\omega^N_0, e_ke_l\>^2 \<\omega^N_0, e_{k'}e_l\>^2 \big) - \frac14 \eps_N^{-4} |l|^4.
  \end{equation}

Now we compute the expectation on the right hand side of \eqref{case2-3}. We have
  $$\<\omega^N_0, e_ke_l\>^2 \<\omega^N_0, e_{k'}e_l\>^2 = \frac1{N^2} \sum_{r,s, r',s'=1}^N \xi_r \xi_s \xi_{r'} \xi_{s'} (e_ke_l)(X^r_0) (e_ke_l)(X^s_0) (e_{k'} e_l)(X^{r'}_0) (e_{k'} e_l)(X^{s'}_0).$$
The Isserlis--Wick theorem implies
  \begin{equation}\label{case2-4} \aligned
  \E \big(\<\omega^N_0, e_ke_l\>^2 \<\omega^N_0, e_{k'}e_l\>^2 \big) =&\, \frac1{N^2} \sum_{r,r'=1}^N \E \big[(e_k^2 e_l^2)(X^r_0) (e_{k'}^2 e_l^2)(X^{r'}_0)\big]\\
  &\, + \frac2{N^2} \sum_{r,s=1}^N \E \big[(e_k e_{k'} e_l^2)(X^r_0) (e_k e_{k'} e_l^2)(X^s_0) \big]\\
  =: &\, J_1 +J_2.
  \endaligned
  \end{equation}
First,
  $$J_1= \frac1{N^2} \sum_{1\leq r\neq r'\leq N} \E\big[(e_k^2 e_l^2)(X^r_0) \big] \E\big[(e_{k'}^2 e_l^2)(X^{r'}_0)\big] + \frac1{N^2} \sum_{r=1}^N \E \big[(e_k^2 e_{k'}^2 e_l^4)(X^r_0)\big] $$
which are denoted by $J_{1,1}$ and $J_{1,2}$, respectively. Note that
  $$J_{1,1}= \bigg(1- \frac1N \bigg) \int e_k^2 e_l^2 \,\d x \int e_{k'}^2 e_l^2 \,\d x$$
and
  $$J_{1,2} = \frac1{N} \int e_k^2 e_{k'}^2 e_l^4\,\d x \leq \frac{16}N.$$
Moreover,
  $$\sum_{k, k'\in \Lambda_N} C_{k,l}^2 C_{k',l}^2 \cdot J_{1,1} = \bigg(1- \frac1N \bigg) \bigg(\sum_{k \in \Lambda_N} C_{k,l}^2 \int e_k^2 e_l^2 \,\d x\bigg)^2 =\frac14 \bigg(1- \frac1N \bigg) \eps_N^{-4} |l|^4,$$
where the last step is due to \eqref{case2-2}. Therefore,
  \begin{equation}\label{case2-5}
  \sum_{k, k'\in \Lambda_N} C_{k,l}^2 C_{k',l}^2 \cdot J_1 = \frac14 \eps_N^{-4} |l|^4 + O\bigg(\frac{(\log N)^2}N \bigg).
  \end{equation}

It remains to estimate $J_2$ in \eqref{case2-4}. Similarly,
  $$ J_2 =\frac2{N^2} \sum_{1\leq r\neq s\leq N} \E \big[(e_k e_{k'} e_l^2)(X^r_0) \big] \E \big[(e_k e_{k'} e_l^2)(X^s_0) \big] + \frac2{N^2} \sum_{r=1}^N \E \big[(e_k^2 e_{k'}^2 e_l^4)(X^r_0) \big]. $$
We write $J_{2,1}$ and $J_{2,2}$ for the two terms. We still have
  $$J_{2,2}= \frac2{N}\int e_k^2 e_{k'}^2 e_l^4\,\d x \leq \frac{32}N.$$
Next,
  $$J_{2,1} = 2 \bigg(1- \frac1N \bigg) \bigg( \int e_k e_{k'} e_l^2 \,\d x \bigg)^2.$$
As a result,
  $$\sum_{k, k'\in \Lambda_N} C_{k,l}^2 C_{k',l}^2 \cdot J_2= 2 \bigg(1- \frac1N \bigg) \sum_{k, k'\in \Lambda_N} C_{k,l}^2 C_{k',l}^2 \bigg( \int e_k e_{k'} e_l^2 \,\d x \bigg)^2 + O\bigg(\frac{(\log N)^2}N \bigg).$$
Note that the sum in the first quantity is equal to $I_N$ defined in \eqref{S2-0}. Therefore,
  $$\bigg| \sum_{k, k'\in \Lambda_N} C_{k,l}^2 C_{k',l}^2 \cdot J_2 \bigg| \leq C_4 + O\bigg(\frac{(\log N)^2}N \bigg).$$
Combining this estimate with \eqref{case2-3}--\eqref{case2-5}, we finally get
  $$\E R_N^2 \leq C_4 + O\bigg(\frac{(\log N)^2}N \bigg).$$
The proof is complete.

\bigskip

\noindent \textbf{Acknowledgements.} The second author is grateful to Professors Rongchan Zhu and Xiangchan Zhu for valuable discussions, and for drawing his attention to the reference \cite{Kur}. He also thanks the financial supports of the National Natural Science Foundation of China (Nos. 11431014, 11571347, 11688101), the Youth Innovation Promotion Association, CAS (2017003) and the Special Talent Program of the Academy of Mathematics and Systems Science, CAS.


\begin{thebibliography}{99}

\bibitem{AC} S. Albeverio, A. B. Cruzeiro, Global flows with invariant (Gibbs) measures for Euler and Navier--Stokes two-dimensional fluids. \textit{Comm. Math. Phys.} \textbf{129} (1990), 431--444.

\bibitem{AF} S. Albeverio, B. Ferrario, Uniqueness of solutions of the stochastic Navier--Stokes equation with invariant measure given by the enstrophy. \emph{Ann. Probab.} \textbf{32} (2004), 1632--1649.

\bibitem{AF2} S. Albeverio and B. Ferrario, Some Methods of Infinite Dimensional Analysis in Hydrodynamics: An Introduction, In SPDE in Hydrodynamic: Recent Progress and Prospects, G. Da Prato and M. R\"{o}ckner Eds., CIME Lectures, Springer--Verlag, Berlin 2008.

\bibitem{Billingsley} P. Billingsley, Convergence of Probability Measures. Second edition. Wiley Series in Probability and Statistics: Probability and Statistics. A Wiley-Interscience Publication. \emph{John Wiley \& Sons, Inc., New York,} 1999.

%\bibitem{CF} M. Coghi, F. Flandoli, Propagation of chaos for interacting particles subject to environmental noise. \emph{Ann. Appl. Probab.} \textbf{26} (2016),  no. 3, 1407--1442.

\bibitem{DaPD} G. Da Prato, A. Debussche, Two-Dimensional Navier--Stokes Equations Driven by a Space--Time White Noise. \emph{J. Funct. Anal.} \textbf{196} (2002), 180--210.

\bibitem{DaPZ} G. Da Prato, J. Zabczyk, Stochastic Equations in Infinite Dimensions. Encyclopedia of Mathematics and its Applications, 44. \emph{Cambridge University Press, Cambridge,} 1992.

%\bibitem{DFV} F. Delarue, F. Flandoli, D. Vincenzi, Noise prevents collapse of Vlasov-Poisson point charges. \textit{Comm. Pure Appl. Math.} \textbf{67} (2014), no. 10, 1700--1736.

\bibitem{Debussche} A. Debussche, The 2D Navier--Stokes equations perturbed by a delta correlated noise. \emph{Probabilistic methods in fluids,}  115--129, \emph{World Sci. Publ., River Edge, NJ,} 2003.

\bibitem{F1} F. Flandoli, Weak vorticity formulation of 2D Euler equations with white noise initial condition. \emph{Comm. Partial Differential Equations} \textbf{43} (2018), 1102--1149.

\bibitem{FL-1} F. Flandoli, D. Luo, $\rho$-white noise solution to 2D stochastic Euler equations. \emph{Probab. Theory Relat. Fields} (2019), https://doi.org/10.1007/s00440-019-00902-8.

\bibitem{FL-2} F. Flandoli, D. Luo, Kolmogorov equations associated to the stochastic 2D Euler equations. \emph{SIAM J. Math. Anal.}, accepted, see arXiv:1803.05654.

\bibitem{FL-3} F. Flandoli, D. Luo, Convergence of transport noise to Ornstein-Uhlenbeck for 2D Euler equations under the enstrophy measure, arXiv:1806.09332.

\bibitem{Kur} T.G. Kurtz, The Yamada--Watanabe--Engelbert theorem for general stochastic equations and inequalities, \emph{Electron. J. Probab.} \textbf{12} (2007), 951--965.

\bibitem{MP} C. Marchioro, M. Pulvirenti, Mathematical theory of incompressible nonviscous fluids, volume 96 of Applied Mathematical Sciences, Springer--Verlag, New York, 1994.

%\bibitem{KRoz} N. V. Krylov, B. L. Rozovski, Stochastic partial differential equations and diffusion processes. \textit{Russian Math. Surveys} \textbf{37} (1982), 81--105.

%\bibitem{Par} E. Pardoux, Stochastic partial differential equations and filtering of diffusion processes. \textit{Stochastics} \textbf{3} (1979), 127--167.

\bibitem{RZZ} M. R\"ockner, R. Zhu, X. Zhu, Restricted Markov uniqueness for the stochastic quantization of $P(\Phi)_2$ and its applications. \emph{J. Funct. Anal.} \textbf{272} (2017),  no. 10, 4263--4303.

\bibitem{Sauer} M. Sauer, $L^1$-Uniqueness of Kolmogorov Operators Associated with Two-Dimensional Stochastic Navier--Stokes Coriolis Equations with Space-Time White Noise. \emph{J. Theor. Probab.} \textbf{29} (2016), 569--589.

%\bibitem{Simon} J. Simon, Compact sets in the space $L^p(0,T; B)$. \emph{Ann. Mat. Pura Appl.} \textbf{146} (1987), 65--96.

\bibitem{Stannat} W. Stannat, A new a priori estimate for the Kolmogorov operator of a 2D-stochastic Navier--Stokes equation. \emph{Infin. Dimens. Anal. Quantum Probab. Relat. Top.} \textbf{ 10}  (2007),  no. 4, 483--497.

%\bibitem{Stannat11} W. Stannat, $L^p$-uniqueness of Kolmogorov operators associated with 2D-stochastic Navier--Stokes--Coriolis equations. \emph{Math. Nachr.} \textbf{284} (2011),  no. 17--18, 2287--2296.

\bibitem{ZhuZhu} Rongchan Zhu, Xiangchan Zhu, Strong-Feller property for Navier--Stokes equations driven by space-time white noise, arXiv:1709.09306.

\end{thebibliography}
\end{document}